\documentclass[a4paper,11pt]{amsart}
\usepackage{amssymb,amsmath,amsthm}
\usepackage{fullpage}
\usepackage{color}
\definecolor{darkred}{rgb}{0.6,0.1,0.1}
\usepackage[colorlinks=true,linkcolor=blue, citecolor=darkred]{hyperref}
\usepackage{eucal}
\usepackage{xcolor}
\usepackage{stackrel, enumerate}
\usepackage{graphicx, bbm, lmodern}
\definecolor{darkblue}{rgb}{0.2,0.2,0.6}

\newcommand{\beq}{\begin{equation} \begin{split}}
\newcommand{\eeq}{\end{split} \end{equation}}

\newcommand\Omg{\Omega}

\makeatletter
\def\section{\@startsection{section}{1}\z@{.9\linespacing\@plus\linespacing}%
	{.7\linespacing} {\fontsize{13}{14}\selectfont\bfseries\centering}}
\def\paragraph{\@startsection{paragraph}{4}%
	\z@{0.3em}{-.5em}%
	{$\bullet$ \ \normalfont\itshape}}

\@addtoreset{equation}{section}
\makeatother

\renewcommand\and{\qquad\text{and}\qquad}

\newcommand\sm{\setminus}

\newcommand{\comm}[1]{}

\def\sfH{\mathsf{H}}

\def\bm1{\mathbbm{1}}

\def\p{\partial}

\def\omg{\omega}

\def\Re{{\rm Re}\,}

\def\arr{\rightarrow}

\newcommand{\sfJ}{\mathsf{J}}

\def\tt{\theta}

\def\lm{\lambda}

\def\p{\partial}

\def\sfH{\mathsf{H}}

\def\dd{{\,\mathrm{d}}}

\def\omg{\omega}

\newcounter{counter_a}

\usepackage[latin1]{inputenc}
\usepackage[T1]{fontenc}

\newcommand{\eg}{{\it e.g.}\,}

\numberwithin{figure}{section}
\numberwithin{equation}{section}
\theoremstyle{plain}
\newtheorem*{thm*}{Theorem}
\newtheorem{thm}{Theorem}[section]

\newtheorem{lem}[thm]{Lemma}
\newtheorem{prop}[thm]{Proposition}

\newtheorem{dfn}[thm]{Definition}
\theoremstyle{remark}
\newtheorem{remark}[thm]{Remark}
\theoremstyle{plain}

\newcommand{\beu}{\begin{equation*}}
\newcommand{\eeu}{\end{equation*}}
\newcommand{\besu}{\begin{equation*}
\begin{aligned}}
\newcommand{\eesu}{\end{aligned}
\end{equation*}}
\newcommand{\bes}{\begin{equation}
\begin{aligned}}
\newcommand{\ees}{\end{aligned}
\end{equation}}

\newcommand\cK{\mathcal K}

\newcommand\cL{\mathcal L}
\newcommand\cM{\mathcal M}

\newcommand\frh{\mathfrak h}

\newcommand\ov{\overline}

\newcommand\wh{\widehat}

\newcommand\void[1]{}

\def\ov{\overline}

   \def\dN{{\mathbb N}}   
      \def\dR{{\mathbb R}}
\def\dS{{\mathbb S}}

   \def\sfH{{\mathsf H}}   
\def\sfJ{{\mathsf J}}

   \def\cK{{\mathcal K}}   \def\cL{{\mathcal L}}
\def\cM{{\mathcal M}}

\renewcommand{\div}{\mathrm{div}\,}

\newcommand{\dom}{\mathrm{dom}\,}

\usepackage[left=2.7cm, right=2.7cm, marginparwidth=2cm, textheight = 24cm ]{geometry}
\usepackage[normalem]{ulem}
\definecolor{DarkGreen}{rgb}{0,0.5,0.1}

\definecolor{DarkBlue}{rgb}{0,0.1,0.5}

\newcommand\soutD{\bgroup\markoverwith
	{\textcolor{DarkGreen}{\rule[.5ex]{2pt}{1pt}}}\ULon}
\newcommand\soutP{\bgroup\markoverwith
	{\textcolor{blue}{\rule[.5ex]{2pt}{1pt}}}\ULon}
\newcommand{\Hm}[1]{\leavevmode{\marginpar{\tiny%
			$\hbox to 0mm{\hspace*{-0.5mm}$\leftarrow$\hss}%
			\vcenter{\vrule depth 0.1mm height 0.1mm width \the\marginparwidth}%
			\hbox to
			0mm{\hss$\rightarrow$\hspace*{-0.5mm}}$\\
			\relax\raggedright #1}}}

\begin{document}

\title[]{Improved inequalities between Dirichlet and Neumann eigenvalues of the biharmonic operator}

\author[V. Lotoreichik]{Vladimir Lotoreichik}
\address[V. Lotoreichik]{Department of Theoretical Physics, Nuclear Physics Institute, 	Czech Academy of Sciences, 25068 \v Re\v z, Czech Republic}
\email{lotoreichik@ujf.cas.cz}



\maketitle

\begin{abstract}
	We prove that the $(k+d)$-th Neumann eigenvalue of the biharmonic operator on a bounded connected $d$-dimensional $(d\ge2)$ Lipschitz domain is not larger than its $k$-th Dirichlet eigenvalue for all $k\in\mathbb{N}$. For a
	special class of domains with symmetries we obtain a stronger inequality. Namely,
	for this class of domains, we prove that 
	the $(k+d+1)$-th Neumann eigenvalue of the biharmonic operator does not exceed  its $k$-th Dirichlet eigenvalue for all $k\in\mathbb{N}$. In particular, in two dimensions, this special class consists of domains having an axis of symmetry.
\end{abstract}

\section{Introduction}\label{s:intro}
Eigenvalue inequalities for differential operators is a classical topic in spectral theory.
The aim of the present paper is to obtain inequalities between Dirichlet and Neumann eigenvalues of the biharmonic operator on a bounded domain in the spirit of the eigenvalue inequality proved by Friedlander~\cite{F91}, which states that the $(k+1)$-th Neumann eigenvalue of the Laplacian does not exceed its $k$-th Dirichlet eigenvalue. Recently, Provenzano ~\cite{Pr19} obtained among other results an inequality of this type for the biharmonic operator and conjectured its improved variant. Our analysis is largely inspired by this conjecture. We obtain an eigenvalue inequality for the biharmonic operator, which significantly improves the inequality in~\cite{Pr19} in the space dimension $d \ge 3$, but it still remains weaker than the conjectured inequality. Moreover, for a class of domains with symmetries we managed to prove the eigenvalue inequality for the biharmonic operator conjectured in~\cite{Pr19}. 

The spectral analysis of the biharmonic operator was initially motivated by applications in classical mechanics in the study of vibration of plates~\cite{R}. It is now an independent mathematical area with many challenging open problems. 
Let us now introduce the biharmonic operators with Dirichlet and Neumann boundary conditions on a bounded domain.
Let $\Omg\subset\dR^d$, $d\ge 2$, be a bounded connected Lipschitz domain. As usual, we denote by $H^2(\Omg)$ the $L^2$-based second-order Sobolev space on $\Omg$ and by $H^2_0(\Omg)$ the closure of $C^\infty_0(\Omg)$ in $H^2(\Omg)$.   The non-negative, densely defined quadratic forms
\begin{equation}\label{eq:forms}
\begin{aligned}
		\frh_{\rm D}[u] &:= \sum_{i,j=1}^d\|\p_{ij} u\|^2_{L^2(\Omg)},\qquad \dom\frh_{\rm D} := H^2_0(\Omg),\\
		\frh_{\rm N}[u] & := 
		\sum_{i,j=1}^d\|\p_{ij} u\|^2_{L^2(\Omg)},\qquad \dom\frh_{\rm N} := H^2(\Omg), 
\end{aligned}
\end{equation}
are closed in the Hilbert space $L^2(\Omg)$ (see \eg~\cite[Theorem 2.1]{BK22}); here we use the abbreviation $\p_{ij}u := \frac{\p^2u}{\p x_i\p x_j}$
for $i,j\in\{1,2,\dots,d\}$. 
The self-adjoint Dirichlet  and Neumann  biharmonic operators $\sfH_{\rm D}$ and $\sfH_{\rm N}$ on $\Omg$ are associated with the quadratic forms $\frh_{\rm D}$ and $\frh_{\rm N}$, respectively, via the first representation theorem~\cite[Theorem VI 2.1]{K}. It is not hard to check via integration by parts that both operators act as the biharmonic operator $\Delta^2$ on functions satisfying appropriate boundary conditions, which are specified in Remark~\ref{rem:BC} below.
The spectra of these operators are purely discrete.  In the two-dimensional setting ($d=2$), the Dirichlet biharmonic operator describes the
vibrations of a clamped plate, while its Neumann counterpart describes the free plate. The literature on the biharmonic operators is quite extensive and it is not possible to make a complete overview.
Many contributions are devoted to spectral isoperimetric inequalities~\cite{AB95,  C11, CL20, K20,  Le23, N95}. Asymptotic expansions of the eigenvalues of the biharmonic operator are considered  in~\cite{FR23, FP23, KN11}. Bounds on the eigenvalues of the biharmonic operator are obtained in~\cite{BF20, CP22, WX07}.

Let us now discuss in detail the results obtained in~\cite{Pr19} by Provenzano.
We denote by $\{\lm_k\}_{k\ge 1}$ and $\{\mu_k\}_{k\ge 1}$ the eigenvalues of $\sfH_{\rm D}$ and $\sfH_{\rm N}$, respectively, enumerated in the non-decreasing order and repeated with multiplicities taken into account.
The following inequality is proved in~\cite[Theorem 1.1]{Pr19}
\begin{equation}\label{eq:Provenzano}
	\mu_{k+2} < \lm_k,\qquad \text{for all}\,\,k\in\dN.
\end{equation}
In~\cite{Pr19}, a more general inequality for the polyharmonic operators on $\Omg$ is actually proved, which reduces to~\eqref{eq:Provenzano} in the case of the biharmonic operator. In the same paper it was conjectured that the inequality~\eqref{eq:Provenzano} can be strengthened as
\begin{equation}\label{eq:conjecture}
	\mu_{k+d+1} \le \lm_k,\qquad \text{for all}\,\, k\in\dN\qquad \text{\rm (conjecture)},
\end{equation}
and that the inequality in~\eqref{eq:conjecture} might even be always strict.

In the present paper we improve the inequality~\eqref{eq:Provenzano} in the space dimension $d\ge 3$ in the general setting and prove the conjectured inequality~\eqref{eq:conjecture} for a class of domains having additional symmetries. Our first result concerns the general case for higher space dimensions.
\begin{thm}\label{thm1}
	For the space dimension $d\ge 3$, the following inequality holds
	\[
		\mu_{k+d} \le \lm_k,\qquad\text{for all}\,\,\,k\in\dN.
	\]
\end{thm}
The proof of the above theorem still works in two dimensions and gives $\mu_{k+2}\le \lm_k$ for all $k\in\dN$, which is slightly weaker than the inequality~\eqref{eq:Provenzano}. For this reason we have excluded the case $d=2$ from the formulation.

In order to formulate the second result
we introduce for $l\in\{2,\dots,d\}$ the mapping 
\begin{equation}\label{eq:symmetries}
	\sfJ_l\colon\dR^d\arr\dR^d,\qquad \sfJ_l x := (x_1,x_2,\dots,-x_l,\dots,x_d)^\top, 
\end{equation}
where $x = (x_1,x_2,\dots,x_d)^\top$.
The domain $\Omg$ is symmetric with respect to the hyperplane $\{x\in\dR^d\colon x_l = 0\}$ with $l\in\{2,\dots,d\}$ in the case that $\sfJ_l(\Omg) = \Omg$.
\begin{thm}\label{thm2}
	Let $d \ge 2$ and assume that
	the bounded connected Lipschitz domain $\Omg$ is such that 
	$\sfJ_l(\Omg) = \Omg$ for all $l\in\{2,\dots,d\}$. Then the following inequality holds
	\[
		\mu_{k+d+1}\le \lm_k,\qquad\text{for all}\,\,\,k\in\dN.
	\]
\end{thm}
The special role of the $x_1$-axis
in the formulation of Theorem~\ref{thm2} is not restrictive. Since we are allowed
to translate and rotate the domain $\Omg$, we only require that $\Omg$ is symmetric with respect to $d-1$ mutually orthogonal hyperplanes. 
In two dimensions, the condition in Theorem~\ref{thm2} is equivalent to the fact that $\Omg$ has an axis of symmetry.
Theorem~\ref{thm2} applies, in particular, to ellipses, rectangles, isosceles triangles, and annuli. In three dimensions, Theorem~\ref{thm2} can be applied, for example, to the cylindrical domain $\Omg = (0,h)\times \omg\subset\dR^3$ of some positive height $h > 0$ and with the cross-section $\omg\subset\dR^2$ being a bounded connected Lipschitz domain with an axis of symmetry. 

The proofs of Theorems~\ref{thm1} and~\ref{thm2} rely on the min-max principle, for which we construct a subspace of trial functions using a modification of the approach suggested by Filonov in~\cite{F05}
for the proof of a similar inequality between Dirichlet and Neumann eigenvalues of the Laplacian. The key idea is to construct a trial subspace for $\sfH_{\rm N}$  
as a span of $k$ orthonormal eigenfunctions of $\sfH_{\rm D}$ corresponding to its first $k$ eigenvalues and of $d$ additional functions for the proof of Theorem~\ref{thm1}
and, respectively, $d+1$ additional functions for the proof of Theorem~\ref{thm2}. These additional functions are constructed in a different way than in the paper~\cite{Pr19}. Our construction better reflects the properties of the biharmonic operator.  
Any non-trivial function $u$ from the span of the additional functions
belongs to $H^2(\Omg)$, is linearly independent from the span of $k$ orthonormal eigenfunctions of $\sfH_{\rm D}$ corresponding to its first $k$ eigenvalues, and satisfies the identities
\begin{equation}\label{eq:properties}
	\Delta^2 u = \lm_k u\qquad\text{and}\qquad\sum_{i,j=1}^d\|\p_{ij}u\|^2_{L^2(\Omg)} = \lm_k\|u\|^2_{L^2(\Omg)}.
\end{equation}
The second property in~\eqref{eq:properties}
is achieved by the construction of the additional functions satisfying a number of
orthogonality relations. In the proof of Theorem~\ref{thm1} we make use of the Borsuk-Ulam theorem to achieve the required orthogonality. In the proof of Theorem~\ref{thm2} these relations are satisfied thanks to the rich symmetry properties of the domain.

In order to position our work in the existing literature we review known results on inequalities between Dirichlet and Neumann eigenvalues.
Such inequalities is a classical topic for the Laplace operator. We will denote by $\{\wh\lm_k\}_{k\ge 1}$ and $\{\wh\mu_k\}_{k\ge 1}$ the eigenvalues of the Laplace operator on a bounded domain with, respectively, Dirichlet and Neumann boundary conditions, enumerated in the non-decreasing order and repeated with multiplicities taken into account.
P\'{o}lya proved
in~\cite{P52} that $\wh\mu_2 < \wh\lm_1$ in two dimensions for smooth domains. For convex, two-dimensional domains with $C^2$-boundary Payne proved
in~\cite{P55} that $\wh\mu_{k+2} < \wh\lm_k$ for all 
$k \in \dN$. The result of Payne was  generalized by Levine and Weinberger \cite{LW86}, who  proved the
inequality $\wh\mu_{k+d} \le \wh\lm_k$ for all $k\in\dN$, for any convex domain. For general (not necessarily convex)
bounded $C^1$-smooth domains in any space dimension the inequality $\wh\mu_{k+1}\le \wh\lm_k$ for all $k\in\dN$ was obtained by Friedlander in~\cite{F91}. Filonov~\cite{F05} simplified the argument by Friedlander and showed that
the strict inequality $\wh\mu_{k+1} < \wh\lm_k$ holds for all $k\in\dN$, for a class of domains even less regular than Lipschitz. 
Similar type inequalities are also later obtained for other types of differential operators: see \cite{FL10} for the sub-Laplacian on the Heisenberg group and the magnetic Laplacian, \cite{LR17} for the Laplacian with mixed boundary conditions,  and \cite{DT22} for the Stokes operator.
The present paper belongs to the same line of research. 

It remains to outline the structure of the paper. In Section~\ref{sec:pre} we collect  preliminary facts and auxiliary lemmas, which will later be used in the proofs of the main results. In the same preliminary section, we also provide more details on the Dirichlet and Neumann spectral problems for the biharmonic operator.  In Section~\ref{sec:proof1} we prove Theorem~\ref{thm1} and, finally, in Section~\ref{sec:proof2} we prove Theorem~\ref{thm2}. The proofs of these two theorems share common ideas, but we prefer to provide both arguments in full detail for the convenience of the reader.
  
\section{Preliminaries}\label{sec:pre}
In this preliminary section we collect a number of tools used in the proofs of the main results. First, in Subsection~\ref{ssec:biharmonic} we recall known properties of the biharmonic operators with Dirichlet and Neumann boundary conditions.
Next, in Subsection~\ref{ssec:unique} we state a lemma based on the unique continuation principle for the Laplace operator.
Finally, in Subsection~\ref{ssec:BU}
we present an auxiliary construction
based on the Borsuk-Ulam theorem of a special family of orthogonal functions. This family will be used in the proof of Theorem~\ref{thm1}.

\subsection{Dirichlet and Neumann biharmonic operators}\label{ssec:biharmonic}
Recall that $\Omg\subset\dR^d$, $d\ge 2$, is a bounded connected Lipschitz domain and that the biharmonic operators $\sfH_{\rm D}$ and $\sfH_{\rm N}$ with Dirichlet and Neumann boundary conditions are associated with the quadratic forms $\frh_{\rm D}$ and $\frh_{\rm N}$, respectively, defined in~\eqref{eq:forms}.
\begin{remark}\label{rem:BC}
	In this remark we will provide strong formulations of the spectral problems for $\sfH_{\rm D}$ and $\sfH_{\rm N}$. 	
	It is not difficult to verify via integration by parts that the spectral problem	for the Dirichlet biharmonic operator $\sfH_{\rm D}$ can be written
	in the strong formulation as
	\begin{equation}\label{eq:Dstrong}
	\begin{cases}
	\Delta^2 u = \lm u,&\qquad \text{in}\,\,\Omg,\\
	u = 0,&\qquad\text{on}\,\,\p\Omg,\\
	\frac{\p u}{\p \nu} = 0,&\qquad\text{on}\,\,\p\Omg,
	\end{cases}
	\end{equation}
	where $\frac{\p}{\p \nu}$ stands for the normal derivative on the boundary with the normal pointing outwards of $\Omg$ and where $\lm$ is the spectral parameter.
	This strong formulation of the spectral problem for $\sfH_{\rm D}$ means that
	a non-trivial function $u\in H^2(\Omg)$ is an eigenfunction of $\sfH_{\rm D}$  if, and only if, it satisfies the system ~\eqref{eq:Dstrong} with a certain value of the spectral parameter $\lm$.
	
	The strong formulation of the spectral problem for the Neumann biharmonic operator involves a rather complicated boundary condition and it was derived in~\cite[Proposition 5]{C11} (see also~\cite{BK22}). Assuming, in addition,
	that $\Omg$ is a $C^\infty$-smooth domain, the spectral problem for $\sfH_{\rm N}$ can be written as
	\begin{equation}\label{eq:Nstrong}
	\begin{cases}
	\Delta^2 u = \mu u,&\qquad \text{in}\,\,\Omg,\\[0.2ex]
	\frac{\p^2u}{\p\nu^2} = 0,&\qquad\text{on}\,\,\p\Omg,\\[0.2ex]
	\frac{\p(\Delta u)}{\p \nu} + \div_{\!\p\Omg}\big(P_{\p\Omg}\big[(D^2 u)\nu\big]\big) = 0,&\qquad\text{on}\,\,\p\Omg,
	\end{cases}
	\end{equation}
	where
	$\frac{\p^2}{\p \nu^2}$ stands for the second-order normal derivative on the boundary, $\div_{\p\Omg}$
	is the surface divergence on $\p\Omg$, $P_{\p\Omg}$ orthogonally projects the vector in $\dR^d$ at $x\in\p\Omg$ into the tangent plane of $\p\Omg$ at $x$, $D^2u$ is the Hessian of the function $u$, $\nu$ stands for the outer unit normal vector for $\Omg$, and $\mu$ is the spectral parameter. It can be checked that any eigenfunction of $\sfH_{\rm N}$ on a bounded $C^\infty$-smooth domain belongs to $C^\infty(\ov\Omg)$ (see~\cite[Proposition 2]{C11}).
   The above strong formulation of the spectral problem for $\sfH_{\rm N}$ on a bounded $C^\infty$-smooth domain means that a non-trivial function $u\in H^2(\Omg)$ is an eigenfunction of $\sfH_{\rm N}$ if, and only if, $u$ belongs to $C^\infty(\ov\Omg)$ and satisfies the equations in~\eqref{eq:Nstrong} with a certain value of the spectral parameter $\mu$.
   We also remark that as an interpolation between Dirichlet and Neumann boundary conditions, the Robin boundary conditions for the biharmonic operator were recently introduced~\cite{BK22, CL20, L23}. 
\end{remark}
Recall that $\{\lm_k\}_{k\ge1}$
and $\{\mu_k\}_{k\ge1}$ denote the eigenvalues of the biharmonic operators $\sfH_{\rm D}$ and $\sfH_{\rm N}$, respectively, enumerated in the non-decreasing order and repeated with multiplicities taken into account.
By the min-max principle~\cite[\S 4.5]{D95} (see also ~\cite[Theorem 1.28]{FLW23}) these eigenvalues can be characterised as follows
\begin{align}
	\label{eq:minmaxD}
	\lm_k &= 
	\inf_{\begin{smallmatrix}\cL\subset H^2_0(\Omg)\\ \dim\cL = k\end{smallmatrix} }\sup_{u\in \cL\sm\{0\}} 
	\frac{\sum_{i,j=1}^d\|\p_{ij} u\|^2_{L^2(\Omg)}}{\|u\|^2_{L^2(\Omg)}},\qquad k\in\dN,\\ 
	\label{eq:minmaxN}\mu_k & = 
	\inf_{\begin{smallmatrix}\cL\subset H^2(\Omg)\\ \dim\cL = k\end{smallmatrix} }\sup_{u\in \cL\sm\{0\}} 
	\frac{\sum_{i,j=1}^d\|\p_{ij} u\|^2_{L^2(\Omg)}}{\|u\|^2_{L^2(\Omg)}},\qquad k\in\dN, 
\end{align}
where the infima are taken with respect to $k$-dimensional linear subspaces of the respective Sobolev spaces. Moreover, the infimum in~\eqref{eq:minmaxD} is attained on the span of $k$ orthonormal eigenfunctions of $\sfH_{\rm D}$ corresponding to the eigenvalues $\lm_1,\lm_2,\dots,\lm_k$ and, in this case, the supremum is attained when $u$ is an eigenfunction of $\sfH_{\rm D}$ corresponding to the eigenvalue $\lm_k$.  
Analogously, the infimum in~\eqref{eq:minmaxN} is attained on the span of $k$ orthonormal eigenfunctions of $\sfH_{\rm N}$ corresponding to the eigenvalues $\mu_1,\mu_2,\dots,\mu_k$ and, in this case, the supremum is attained when $u$ is an eigenfunction of $\sfH_{\rm N}$ corresponding to $\mu_k$.
In particular, for the Dirichlet biharmonic operator we get that for any $k\in\dN$ there exists a linear subspace $\cL\subset H^2_0(\Omg)$ with $\dim\cL = k$ such that
\begin{equation}\label{eq:cL}
	\frh_{\rm D}[u] \le \lm_k\|u\|^2_{L^2(\Omg)},\qquad\text{for all}\,\,u\in\cL.
\end{equation}
This linear subspace $\cL$ is constructed as a span of $k$ orthonormal eigenfunctions of $\sfH_{\rm D}$ corresponding to its first $k$ eigenvalues.
\begin{remark}
	Let us introduce the characteristic function $\chi_\Omg$ of $\Omg$
	and the functions $f_k(x) := x_k$ for $k\in \{1,2,\dots,d\}$. Clearly, the family of functions $\{\chi_\Omg,f_1,f_2,\dots,f_d\}$ is linearly independent in $L^2(\Omg)$. According to \cite[Theorem 7]{Pr18} we have the following characterisation
	\[
		\ker\sfH_{\rm N} = {\rm span}\,\big\{\chi_\Omg,f_1,f_2,\dots, f_d\big\}.
	\]
	Hence, we conclude that $\mu_1 = \mu_2 = \dots = \mu_{d+1} = 0$ and that $\mu_{d+2} > 0$. It can also be easily checked that $\lm_1 > 0$.
	For $k = 1$ the inequality in Theorem~\ref{thm1} 
	naturally follows from these simple observations. The inequality in Theorem~\ref{thm1} becomes non-trivial for $k > 1$.   
\end{remark}
\subsection{A unique continuation lemma}
\label{ssec:unique}
In this short subsection we recall a useful result based on the unique continuation for the Laplace operator. 
\begin{lem}\cite[Proposition 2.5]{BR12}\label{lem:BR}
	Let $\Omg\subset\dR^d$ be a bounded connected Lipschitz domain.
	Let $\lm\in\dR$ and let $u\in H^2_0(\Omg)$ be such that $-\Delta u =\lm u$ in $\Omg$. Then $u\equiv 0$.
\end{lem}
This lemma is essentially equivalent to the fact no eigenfunction of the Dirichlet Laplacian on $\Omg$ can simultaneously satisfy the Neumann boundary condition on $\p\Omg$. It will be used to show that the
additional functions in the construction of the trial subspace for the min-max principle in the proofs of the main results are linearly independent from the
spans of the eigenfunctions of the Dirichlet biharmonic operator.
\subsection{A construction of special orthogonal functions}
\label{ssec:BU}
We will present a construction of $d$
(dimension of $\Omg$) mutually orthogonal functions in the Hilbert space $L^2(\Omg)$. This construction is used in the proof of Theorem~\ref{thm1}. To this aim we recall the definition of an odd
continuous function on the $n$-dimensional unit sphere $\dS^n\subset\dR^{n+1}$, $n\in\dN$, centred at the origin.
\begin{dfn}\label{def:odd}
	For $n\in\dN$, a continuous function $g\colon\dS^n\arr\dR^n$ is called \emph{odd} if $g(\tt) = -g(-\tt)$ for any
	$\tt\in\dS^n$.
\end{dfn}
Next, we will formulate the Borsuk-Ulam theorem for odd functions.
\begin{prop}\cite[Theorem 2.1.1]{M03}
	\label{prop:Borsuk}
	For any $n\in\dN$ and any continuous odd function $g\colon \dS^n\arr\dR^n$ there exists a point $\tt_\star\in\dS^n$ such that $g(\tt_\star) = 0$.
\end{prop}
Now we are ready to present the construction of the orthogonal functions. In the formulation and the proof of this lemma we will use the standard notation $x\cdot y = \sum_{j=1}^d x_j y_j$ for
the scalar product in $\dR^d$ of vectors $x = (x_1,x_2,\dots, x_d)$ and $y = (y_1,y_2,\dots,y_d)$.
\begin{lem}\label{lem:aux}
	Let $\Omg\subset\dR^d$, $d\ge 2$, be a bounded connected Lipschitz domain.
	For any $\lm > 0$ one can find $d$ vectors $\omg_l\in\dR^d$ with $|\omg_l|^4 = \lm$ for $l \in\{1,2,\dots,d\}$ such that the functions
	\[
		v_l(x) := \sin(\omg_l\cdot x)\in L^2(\Omg),\qquad l\in\{1,2,\dots,d\} 
	\]
	satisfy the orthogonality relations
	\begin{equation}\label{eq:ortho}
		\int_\Omg v_i(x) v_j(x) \dd x = 0,\qquad i\ne j.
	\end{equation}
\end{lem}
\begin{proof}
	The construction proceeds in $d$ steps. In the first step we take an arbitrary vector $\omg_1\in\dR^d$ satisfying $|\omg_1|^4 = \lm$
	and by that we also fix the function $v_1(x) :=\sin(\omg_1\cdot x)\in L^2(\Omg)$.
	
	In the second step, in order to construct the second vector $\omg_2$ we consider the function 
	\[
		g_2\colon\dS^{d-1}\arr\dR^{d-1},\qquad g_2(\tt) := 
		\big(\big(v_1,\sin(\lm^{\frac14}(\tt\cdot x)\big)_{L^2(\Omg)},
		0,0,	\dots,
		0)^\top.
	\]
	By its construction the function $g_2$ is a continuous odd function in the sense of Definition~\ref{def:odd} (with $n = d-1$). Indeed,
	we clearly get by the Lebesgue dominated convergence theorem that
	\[
		\int_\Omg v_1(x)\sin(\lm^{\frac14}(\tt\cdot x))\dd x \arr \int_\Omg v_1(x)\sin(\lm^{\frac14}(\tt'\cdot x))\dd x,\qquad \text{as}\,\,\tt\arr\tt'\,\, \text{in}\,\, \dS^{d-1},
	\]
	because the integrands converge pointwise and the characteristic function $\chi_\Omg$ of $\Omg$ is the integrable majorant. The relation $g_2(\tt) = -g_2(-\tt)$ follows from the fact that $\sin(\lm^\frac14(-\tt\cdot x)) = -\sin(\lm^\frac14(\tt\cdot x))$.
	Hence, by Proposition~\ref{prop:Borsuk} there exists a point $\tt_2\in\dS^{d-1}$ such that $g_2(\tt_2) = 0$. We set $\omg_2 = \tt_2\lm^{\frac14}$
	and in this way the function $v_2(x) = \sin(\omg_2\cdot x)$ is also fixed.
	
	In the $l$-th step ($l\le d$), in order to construct the vector $\omg_l$ we consider the function $g_l\colon\dS^{d-1}\arr\dR^{d-1}$
	defined by
	\[
		g_l(\tt) := \Big((v_1,
		\sin(\lm^\frac14(\tt\cdot x))_{L^2(\Omg)},
		(v_2,
		\sin(\lm^\frac14(\tt\cdot x))_{L^2(\Omg)}\dots,
		(v_{l-1},
		\sin(\lm^\frac14(\tt\cdot x))_{L^2(\Omg)},0,\dots,0\Big)^\top.
	\]
	Analogously we find that $g_l$ is a continuous odd function in the sense of Definition~\ref{def:odd}. Hence, by Proposition~\ref{prop:Borsuk}
	there exists $\tt_l\in\dS^d$ such that $g_l(\tt_l) = 0$. We set $\omg_l := \lm^\frac14\tt_l$ and define $v_l(x) :=\sin(\omg_l\cdot x)$. We repeat this step until $l = d$. The orthogonality conditions in~\eqref{eq:ortho} are satisfied
	thanks to the choice of the points $\tt_l \in\dS^{d-1}$, $l\in\{2,3,\dots,d\}$.
\end{proof}

\begin{remark}
	It is clearly seen from the proof of Lemma~\ref{lem:aux} that it is not, in general, possible to construct more than $d$ orthogonal functions of the same type
	using our approach based on the Borsuk-Ulam theorem.
\end{remark}

\section{Proof of Theorem~\ref{thm1}}\label{sec:proof1}
The proof is divided into two steps. In the first step we construct a trial subspace of functions, which is then used in the second step together with the min-max principle.
  
\noindent{\it Step 1: construction of a trial subspace.}
Let $k\in\dN$ and let us fix the shorthand abbreviation $\lm = \lm_k$ for the $k$-th eigenvalue of the biharmonic operator $\sfH_{\rm D}$ on $\Omg$ with Dirichlet boundary conditions.
By the consequence~\eqref{eq:cL} of the min-max principle there is a $k$-dimensional linear subspace $\cL_\lm$ of $H^2_0(\Omg)$ such that
\begin{equation}\label{eq:cL2}
	\frh_{\rm D}[u] \le \lm \|u\|^2_{L^2(\Omg)},\qquad \text{for all}\,\,u\in\cL_\lm.
\end{equation}
Let the vectors $\omg_l\in\dR^d$, $|\omg_l|^4 = \lm$, for $l\in\{1,\dots,d\}$ and the associated auxiliary  functions $v_l(x)=\sin(\omg_l\cdot x)\in H^2(\Omg)$, $l\in\{1,\dots,d\}$, be constructed as in Lemma~\ref{lem:aux} with $\lm$ being,  as above, the $k$-th eigenvalue of the Dirichlet biharmonic operator. Recall also that the functions $\{v_l\}_{l=1}^d$ are constructed to be mutually orthogonal in $L^2(\Omg)$. Consider the linear subspace of $H^2(\Omg)$ defined by
\begin{equation}\label{eq:cM}
	\cM := \cL_\lm + {\rm span}\,\big\{v_1,v_2,\dots,v_d\big\}.
\end{equation}
In order to show that $\dim\cM = k + d$ we need to verify that
\[	
	\cL_\lm\cap{\rm span}\,\{v_1,v_2,\dots,v_d\} = \{0\}.
\]
A generic element of $\cL_\lm\cap{\rm span}\,\{v_1,v_2,\dots,v_d\}$ is given by
\[
	w = \sum_{l=1}^d c_l v_l
\]
with some complex coefficients $\{c_l\}_{l=1}^d$ and satisfies $w\in\cL_\lm\subset H^2_0(\Omg)$. Using the definition of the functions $\{v_l\}_{l=1}^d$ we get for any $x\in\Omg$ by a direct computation
\[
\begin{aligned}
	(-\Delta w)(x) & 
	= 
	-\sum_{l=1}^d c_l\Delta(\sin(\omg_l\cdot x))\\
	& = \sum_{l=1}^d |\omg_l|^2 c_l \sin(\omg_l\cdot x)
	= \sqrt{\lm}\sum_{l=1}^d c_l v_l(x) =\sqrt{\lm} w(x).
\end{aligned}	
\]
Thus, we have $w\in H^2_0(\Omg)$ and $-\Delta w = \sqrt{\lm}w$ and by Lemma~\ref{lem:BR} we infer that $w \equiv 0$.

\medskip

\noindent{\it Step 2: application of the min-max principle.}
Recall that $\cM$ is a linear subspace of $H^2(\Omg)$ with $\dim\cM = k+d$ constructed in the previous step of the proof.
Let $v\in \cM$ be arbitrary. The function $v$ is represented by 
\[
	v = u + \sum_{l=1}^d c_l v_l,
\]  
with a certain function $u\in\cL_\lm$ and  some complex coefficients $\{c_l\}_{l=1}^d$. Plugging the function $v$ into the quadratic form $\frh_{\rm N}$ of the biharmonic operator $\sfH_{\rm N}$ with Neumann boundary conditions we get
\begin{equation}\label{eq:0}
	\frh_{\rm N}[v] = \sum_{i,j=1}^d \|\p_{ij} u\|^2_{L^2(\Omg)} + 2\Re\left\{\sum_{i,j=1}^d\sum_{l=1}^d\left(\p_{ij} u,  c_l \p_{ij}v_l\right)_{L^2(\Omg)}\right\} + \sum_{i,j=1}^d\bigg\|\sum_{l=1}^dc_l\p_{ij}v_l\bigg\|^2_{L^2(\Omg)}. 
\end{equation}
We analyse the three terms on the right hand side of the above equation separately.
By~\eqref{eq:cL2} we get that
\begin{equation}\label{eq:1}
	\sum_{i,j=1}^d \|\p_{ij} u\|_{L^2(\Omg)}^2 
	\le 
	\lm\|u\|_{L^2(\Omg)}^2.
\end{equation}

Notice that $\p_j u \in H^1_0(\Omg)$ for any $j\in\{1,2,\dots,d\}$. Hence, we can apply the integration by parts formula in~\cite[Theorem 1.5.3.1]{G85} twice to get
\begin{equation}\label{eq:2}
\begin{aligned}
\sum_{i,j=1}^d\sum_{l=1}^d\left(\p_{ij} u,  c_l \p_{ij}v_l\right)_{L^2(\Omg)} &\!=\!
-\sum_{i,j=1}^d\sum_{l=1}^d \left(\p_j u,c_l\p_{i}\p_{ij}v_l\right)_{L^2(\Omg)}
\!=\! \sum_{i,j=1}^d\sum_{l=1}^d \left(u,c_l\p_{ij}\p_{ij}v_l\right)_{L^2(\Omg)}\\
&=\sum_{l=1}^d\sum_{i,j=1}^d \ov{c_l}\left(u, \omg_{l,i}^2\omg_{l,j}^2 v_l\right)_{L^2(\Omg)}
=\lm\sum_{l=1}^d \ov{c_l}(u,v_l)_{L^2(\Omg)},
\end{aligned}
\end{equation}
where we used the notation $\omg_l = (\omg_{l,1},\omg_{l,2},\dots,\omg_{l,d})^\top$ for $l\in\{1,2,\dots,d\}$ and applied that $|\omg_l|^4 = \lm$ in the last step.

For the last term in~\eqref{eq:0} we obtain
employing mutual orthogonality of $\{v_l\}_{l=1}^d$ in $L^2(\Omg)$ that
\begin{equation}\label{eq:3}
\begin{aligned}
	\sum_{i,j=1}^d\bigg\|\sum_{l=1}^d c_l\p_{ij}v_l\bigg\|^2_{L^2(\Omg)} 
	&=
	\sum_{i,j=1}^d\bigg\|-\sum_{l=1}^dc_l\omg_{l,i}\omg_{l,j}v_l\bigg\|_{L^2(\Omg)}^2
	= 
	\sum_{l=1}^d\sum_{i,j=1}^d
	\omg^2_{l,i}\omg^2_{l,j}|c_l|^2\|v_l\|_{L^2(\Omg)}^2\\
	& = \lm\sum_{l=1}^d|c_l|^2\|v_l\|^2_{L^2(\Omg)} = \lm\bigg\|\sum_{l=1}^d c_lv_l\bigg\|^2_{L^2(\Omg)} ,
\end{aligned}
\end{equation}
where we used that $|\omg_l|^4 = \lm$ for all $l\in\{1,2,\dots,d\}$ in between.

Combining~\eqref{eq:0} with~\eqref{eq:1},~\eqref{eq:2}, and~\eqref{eq:3} we obtain that
\[
\begin{aligned}
	\frh_{\rm N}[v] &\le \lm\|u\|^2_{L^2(\Omg)} + 2\lm\Re\left\{\bigg(u,\sum_{l=1}^d c_lv_l\bigg)_{L^2(\Omg)}\right\} + \lm\bigg\|\sum_{l=1}^d c_l v_l\bigg\|^2_{L^2(\Omg)}\\
	& = \lm\bigg\|u+\sum_{l=1}^dc_l v_l\bigg\|^2_{L^2(\Omg)} = \lm\|v\|^2_{L^2(\Omg)}.
\end{aligned}
\]  
Finally, taking into account that $\dim\cM = k+d$, we get from the inequality $\frh_{\rm N}[v]\le \lm\|v\|^2_{L^2(\Omg)}$ valid for any $v\in\cM$ combined with the min-max characterisation~\eqref{eq:minmaxN} for the eigenvalues of the biharmonic operator with Neumann boundary conditions that $\mu_{k+d}\le \lm_k$ for any $k\in\dN$.  
\section{Proof of Theorem~\ref{thm2}}\label{sec:proof2}
Recall that in the assumptions of the theorem the bounded connected Lipschitz domain $\Omg\subset\dR^d$ is such that 
$\sfJ_l(\Omg) = \Omg$ for all $l\in \{2,\dots,d\}$, where the mappings $\sfJ_l$ are defined in~\eqref{eq:symmetries}.
The proof of this theorem relies on the same technique as the proof of Theorem~\ref{thm1} and we again divide the argument into two steps, where in the first step we construct a subspace of trial functions and in the second step we use this subspace together with the min-max principle. 

\medskip

\noindent{\it Step 1: construction of a trial subspace.}
As in the proof of Theorem~\ref{thm1}, let $k\in\dN$ and let us fix the shorthand abbreviation $\lm = \lm_k$ for the $k$-th eigenvalue of the biharmonic operator $\sfH_{\rm D}$ on $\Omg$ with Dirichlet boundary conditions. As before by the consequence~\eqref{eq:cL} of the min-max principle there is a $k$-dimensional linear subspace $\cL_\lm$ of $H^2_0(\Omg)$ such that
\begin{equation}\label{eq:cL3}
	\frh_{\rm D}[u] \le \lm \|u\|^2_{L^2(\Omg)},\qquad \text{for all}\,\,u\in\cL_\lm.
\end{equation}
Let us fix $\omg := \lm^{1/4} > 0$ (scalar) and introduce the following $d + 1$ functions
in the Sobolev space $H^2(\Omg)$
\begin{equation}\label{eq:v}
	v_0(x) := \sin(\omg x_1),\quad 
	v_1(x) := \cos(\omg x_1)\qquad\text{and}\qquad v_l(x) := \sin(\omg x_l),\quad l\in \{2,\dots,d\}.
\end{equation}
We claim that the functions $\{v_0,v_1,\dots v_d\}$ are linearly independent. Indeed, suppose that for some complex numbers $\{c_l\}_{l=0}^d$ we have the relation
\[
	 \sum_{l=0}^d c_l v_l(x) = 0.
\]
Differentiating the above identity with respect to $x_l$ for all $l\in\{1,2,\dots,d\}$ we get that 
for any $x = (x_1,\dots x_d)\in\Omg$ there holds 
\[
	c_0\cos(\omg x_1) - c_1\sin(\omg x_1) = 0\qquad\text{and}\qquad c_l\cos(\omg x_l) = 0\quad \text{for all}\,\,l\in\{2,\dots,d\}.
\]
Thus, we conclude by simple algebraic reasons that $c_l = 0$ for all $l\in\{0,1,\dots,d\}$.

Next, we show that the functions $\{v_l\}_{l=0}^d$ satisfy certain orthogonality properties in $L^2(\Omg)$.
It follows from the symmetries of the domain $\Omg$ that for any $l\in\{2,\dots,d\}$ there holds 
\[
	(v_0,v_l)_{L^2(\Omg)} = 
	\int_\Omg \sin(\omg x_1)\sin(\omg x_l)\dd x = 
	\int_\Omg \sin(\omg y_1)\sin(-\omg y_l)\dd y = -(v_0,v_l)_{L^2(\Omg)},
\]
where we performed the change of variables $x = \sfJ_l y$ in the integral and used that $\sfJ_l(\Omg) = \Omg$.	Hence, we obtain the following orthogonality property
\begin{equation}\label{eq:ortho1}
	(v_0,v_l)_{L^2(\Omg)} = 0,\qquad\text{for all}\,\,l\in\{2,\dots,d\}.
\end{equation}
Analogously we arrive at the orthogonality 
properties
\begin{align}
		(v_1,v_l)_{L^2(\Omg)}& = 0,\qquad \text{for all}\,\,\, l\in\{2,\dots,d\},\label{eq:ortho2}\\
		(v_i,v_j)_{L^2(\Omg)} &=0,\qquad
		\text{for all}\,\, i,j\in\{2,\dots,d\},\,\,i\ne j.
		\label{eq:ortho3}
\end{align}
We also remark that the functions $v_0$ and $v_1$ are, in general, not orthogonal in $L^2(\Omg)$. Their orthogonality is not needed for the argument.

Let us consider the linear subspace of $H^2(\Omg)$ defined by
\[
	\cK := \cL_\lm + {\rm span}\,\{v_0,v_1,\dots,v_d\}.
\]
In order to show that $\dim\cK = k + d+1$ it suffices to check that
\[
	\cL_\lm\cap{\rm span}\,\{v_0,v_1,\dots,v_d\} = \{0\}.
\]
A generic element of $\cL_\lm\cap{\rm span}\,\{v_0,v_1,\dots,v_d\}$ is given by
\[
	w = \sum_{l=0}^d c_l v_l
\]
with some complex coefficients $\{c_l\}_{l=0}^d$ and satisfies $w\in\cL_\lm\subset H^2_0(\Omg)$.
By the choice of the functions $\{v_l\}_{l=0}^d$ we have
$-\Delta w = \sqrt{\lm}w$. 
Thus, taking into account that $w\in H^2_0(\Omg)$ we infer by Lemma~\ref{lem:BR}  that $w \equiv 0$.

\medskip 
 
\noindent{\it Step 2: application of the min-max principle.}
Recall that $\cK$ is a linear subspace of $H^2(\Omg)$ with $\dim\cK = k+d+1$ constructed in the previous step of the proof.
Let $v\in \cK$ be arbitrary. The function $v$ is represented by 
\[
v = u + \sum_{l=0}^d c_l v_l,
\]  
with a certain function $u\in\cL_\lm$ and  some complex coefficients $\{c_l\}_{l=0}^d$. Plugging the function $v$ into the quadratic form $\frh_{\rm N}$ of the biharmonic operator $\sfH_{\rm N}$ with Neumann boundary conditions we get
\begin{equation}\label{eq:02}
\frh_{\rm N}[v] = \sum_{i,j=1}^d \|\p_{ij} u\|^2_{L^2(\Omg)} + 2\Re\left\{\sum_{i,j=1}^d\sum_{l=0}^d\left(\p_{ij} u,  c_l \p_{ij}v_l\right)_{L^2(\Omg)}\right\} + \sum_{i,j=1}^d\bigg\|\sum_{l=0}^dc_l\p_{ij}v_l\bigg\|^2_{L^2(\Omg)}. 
\end{equation}
We recall that by~\eqref{eq:cL2} there holds
\begin{equation}\label{eq:12}
\sum_{i,j=1}^d \|\p_{ij} u\|_{L^2(\Omg)}^2 
\le 
\lm\|u\|_{L^2(\Omg)}^2.
\end{equation}

Notice that $\p_j u \in H^1_0(\Omg)$ for any $j\in\{1,2,\dots,d\}$. Using the explicit form of the function $\{v_l\}_{l=0}^d$ in~\eqref{eq:v} and applying the integration by parts formula~\cite[Theorem 1.5.3.1]{G85} we get
\begin{equation}\label{eq:22}
\begin{aligned}
	\sum_{i,j=1}^d\sum_{l=0}^d\left(\p_{ij} u,  c_l \p_{ij}v_l\right)_{L^2(\Omg)} 
	&=
	\left(\p_1^2 u, c_0\p_1^2 v_0+c_1\p_1^2v_1\right)_{L^2(\Omg)} +
	\sum_{l=2}^d \left(\p_{l}^2u,c_l\p_{l}^2v_l\right)_{L^2(\Omg)}\\
	&=  
	\left( u, c_0\p_1^2\p_1^2 v_0+c_1\p_1^2\p_1^2v_1\right)_{L^2(\Omg)} +
	\sum_{l=2}^d \left(u,c_l\p_{l}^2\p_{l}^2v_l\right)_{L^2(\Omg)}\\
	&= \omg^4(u,c_0 v_0+c_1v_1)_{L^2(\Omg)} + \omg^4\sum_{l=2}^d
	(u,c_l v_l)_{L^2(\Omg)}\\
	& = \lm\sum_{l=0}^d\ov{c_l}(u,v_l)_{L^2(\Omg)},
\end{aligned}
\end{equation}
where we employed that $\lm = \omg^4$; here we use the abbreviation
$\p_i^2 u = \frac{\p^2 u}{\p x_i^2}$, $i\in\{1,2,\dots,d\}$.   

For the last term in~\eqref{eq:02} we obtain
employing the explicit form of the functions $\{v_l\}_{l=0}^d$ in~\eqref{eq:v} and the orthogonality properties~\eqref{eq:ortho1},~\eqref{eq:ortho2} and~\eqref{eq:ortho3} that
\begin{equation}\label{eq:32}
\begin{aligned}
\sum_{i,j=1}^d\bigg\|\sum_{l=0}^d c_l\p_{ij}v_l\bigg\|^2_{L^2(\Omg)} 
& = \big\|c_0 \p_1^2 v_0 + c_1\p_1^2v_1\big\|_{L^2(\Omg)}^2
+ \sum_{l=2}^d|c_l|^2\big\|\p_l^2 v_l\big\|^2_{L^2(\Omg)} \\
&= \omg^4\big\|c_0  v_0 + c_1v_1\big\|_{L^2(\Omg)}^2 + \omg^4\sum_{l=2}^d|c_l|^2\|v_l\|^2_{L^2(\Omg)}\\
& = \lm\bigg\|\sum_{l=0}^d c_l v_l\bigg\|^2_{L^2(\Omg)},
\end{aligned}
\end{equation}
where we used that $\omg^4 = \lm$ in between.

Combining~\eqref{eq:02} with~\eqref{eq:12},~\eqref{eq:22}, and~\eqref{eq:32} we obtain that
\[
\begin{aligned}
\frh_{\rm N}[v] &\le \lm\|u\|^2_{L^2(\Omg)} + 2\lm\Re\left\{\bigg(u,\sum_{l=0}^d c_lv_l\bigg)_{L^2(\Omg)}\right\} + \lm\bigg\|\sum_{l=0}^d c_l v_l\bigg\|^2_{L^2(\Omg)}\\
& = \lm\bigg\|u+\sum_{l=0}^dc_l v_l\bigg\|^2_{L^2(\Omg)} = \lm\|v\|^2_{L^2(\Omg)}.
\end{aligned}
\]  
Finally, taking into account that $\dim\cK = k+d+1$, we derive from the inequality $\frh_{\rm N}[v]\le \lm\|v\|^2_{L^2(\Omg)}$ valid for any $v\in\cK$ combined with the min-max characterisation~\eqref{eq:minmaxN} for the eigenvalues of the biharmonic operator with Neumann boundary conditions that $\mu_{k+d+1}\le \lm_k$ for any $k\in\dN$.  
\section*{Acknowledgement}
The author gratefully acknowledges the support by the grant No.~21-07129S of the Czech Science Foundation.

\newcommand{\etalchar}[1]{$^{#1}$}


\begin{thebibliography}{\textsc{BCD{\etalchar{+}}72}}

\bibitem[AB95]{AB95}
M. S. Ashbaugh and R. Benguria, On Rayleigh's conjecture for the clamped plate and its generalization to three
dimensions, {\it Duke Math. J.} {\bf 78} (1995), 1--17.

	
	
\bibitem[BR12]{BR12}
J. Behrndt and J. Rohleder,
An inverse problem of Calder\'{o}n type with partial data,
{\it Commun. Partial Differ. Equations} {\bf 37} (2012), 1141--1159.

\bibitem[BF20]{BF20}
D.~Buoso and P. Freitas, Extremal eigenvalues of the Dirichlet biharmonic operator on rectangles, {\it Proc. Amer. Math. Soc.} {\bf 148} (2020), 1109--1120.


\bibitem[BK22]{BK22}
D.~Buoso and J.~Kennedy, 
The Bilaplacian with Robin boundary conditions,
{\it SIAM J. Math. Anal.}
{\bf 54} (2022),  36--78.

\bibitem[C11]{C11}
L. M.~Chasman,
An isoperimetric inequality for fundamental tones of free plates, 
{\it Commun. Math. Phys.} {\bf 303}  (2011), 421--449.

\bibitem[CL20]{CL20}
L. M. Chasman and J. Langford, A sharp isoperimetric inequality for the second eigenvalue of the Robin
plate, {\it J. Spectr. Theory}
{\bf 12} (2022),  617--657.

\bibitem[CP22]{CP22}
B.~Colbois and L. Provenzano, Neumann eigenvalues of the biharmonic operator on domains: geometric bounds and related results, {\it J. Geom. Anal.} {\bf 32} (2022), 218.

\bibitem[D95]{D95}
E.\, B.~Davies, \emph{Spectral theory and differential operators},
Cambridge University Press, Cambridge, 1995.

\bibitem[DT22]{DT22}
C.~Denis and A.\,F.\,M.~ter Elst, A Friedlander type estimate for Stokes operators, \texttt{arXiv:2203.12070}.

\bibitem[FR23]{FR23}
V.~Felli and G. Romani, Perturbed eigenvalues of polyharmonic operators in domains with small holes, {\it Calc. Var. Partial Differential Equations} {\bf 62} (2023), 1--36.

\bibitem[FP23]{FP23}
F.~Ferraresso and L.~Provenzano, On the eigenvalues of the biharmonic operator with Neumann boundary conditions on a thin set,
{\it to appear in Bull. London Math. Soc.},
\texttt{arXiv:2108.03969}.

\bibitem[F05]{F05}
N.~Filonov, On an inequality between Dirichlet and Neumann eigenvalues for the Laplace operator,
{\it St. Petersburg Math. J.} {\bf 16}  (2005), 413--416.

\bibitem[FL10]{FL10}
R.\,L.~Frank and A.~Laptev, Inequalities between Dirichlet and Neumann eigenvalues on the
Heisenberg group, {\it Int. Math. Res. Not. IMRN} (2010), 2889--2902.


\bibitem[FLW23]{FLW23}
R.\,L.~Frank, A.~Laptev, and T.~Weidl, 
\emph{Schr\"{o}dinger operators: eigenvalues and Lieb-Thirring inequalities},
Cambridge University Press, Cambridge, 2023.



\bibitem[Fr91]{F91}
L.~Friedlander, 
Some inequalities between Dirichlet and Neumann eigenvalues,
{\it Arch. Ration. Mech. Anal.} {\bf 116} (1991), 153--160.

\bibitem[G85]{G85}
P.~Grisvard, 
\emph{Elliptic problems in nonsmooth domains}, Pitman Publishing, Boston-London-Melbourne, 1985.

\bibitem[K]{K} 
T.~Kato, 
\emph{Perturbation theory for linear operators. Reprint of the 1980 edition}, 
Springer-Verlag, Berlin, 1995. 

\bibitem[KN11]{KN11}
V.~Kozlov and S. Nazarov, The spectrum asymptotics for the Dirichlet problem in the case of the biharmonic operator in a domain with highly indented boundary, 
{\it St. Petersburg Math. J.} {\bf 22} (2011), 941--983.

\bibitem[K20]{K20}
A.~Krist\'{a}ly, Fundamental tones of clamped plates in nonpositively curved spaces, {\it Adv. Math.} {\bf 367} (2020), 107113.

\bibitem[LW86]{LW86}
H.\,A.~Levine and H.\,F.~Weinberger, Inequalities between Dirichlet and Neumann eigenvalues,
{\it Arch. Rational Mech. Anal.} {\bf 94} (1986), 193--208.

\bibitem[Le23]{Le23}
R.~Leylekian, Sufficient conditions yielding the Rayleigh Conjecture for the clamped plate, \texttt{arXiv:2302.06313}.

\bibitem[L23]{L23}
V.~Lotoreichik, An isoperimetric inequality for the perturbed Robin bi-Laplacian in a planar exterior domain, \emph{J. Differential Equations}
\textbf{345} (2023), 285--313. 


\bibitem[LR17]{LR17}
V.~Lotoreichik and J.~Rohleder,
Eigenvalue inequalities for the Laplacian with mixed boundary conditions,
{\it J. Differential Equations} {\bf 263} (2017), 491--508.


\bibitem[M03]{M03}
J.~Matou\v{s}ek, {\it Using the Borsuk-Ulam theorem. Lectures on topological methods in combinatorics and geometry},
Springer, Berlin, 2003.


\bibitem[N95]{N95}
N. S. Nadirashvili, Rayleigh's conjecture on the principal frequency of the clamped plate, {\it Arch. Rational Mech.
	Anal.} {\bf 129} (1995), 1--10.

\bibitem[P55]{P55}
L.\,E.~Payne, Inequalities for eigenvalues of membranes and plates, {\it J. Rational Mech. Anal.} {\bf 4}
(1955), 517--529.

\bibitem[P52]{P52}
G. P\'{o}lya, Remarks on the foregoing paper, {\it J. Math. Phys.} {\bf 31} (1952), 55--57.


\bibitem[Pr18]{Pr18}
L.~Provenzano, A note on the Neumann eigenvalues of the biharmonic operator, {\it Math. Methods Appl. Sci.}
{\bf	41} (2018), 1005--1012.

\bibitem[Pr19]{Pr19}
L.~Provenzano, 
Inequalities between Dirichlet and Neumann eigenvalues of the polyharmonic operators, 
{\it Proc. Amer. Math. Soc.} {\bf 147}  (2019), 4813--4821.

\bibitem[R]{R}
J.\,W.\,S.~Rayleigh, \emph{The  theory  of  sound},   Macmillan,  London, 1877.

\bibitem[WX07]{WX07}
Q.~Wang and C. Xia, Universal bounds for eigenvalues of the biharmonic operator on Riemannian manifolds, {\it J. Funct. Anal.} {\bf 245} (2007), 334--352.
\end{thebibliography}
\end{document}